\documentclass[11pt,a4paper]{article}
\usepackage{amsmath,amssymb,latexsym,amsthm}
\usepackage{mathrsfs}
\usepackage[english]{babel}
\usepackage[latin2]{inputenc}
\theoremstyle{definition} 
\newtheorem{theor}{Theorem}[section]
\newtheorem{lem}[theor]{Lemma}

\newtheorem{prop}[theor]{Proposition}
\newtheorem{definition}{Definition}[section]

\begin{document}
\title{Tangent prolongation of $\mathcal{C}^r$-differentiable loops}
\author{\'Agota Figula (Debrecen), and P\'eter T. Nagy (Budapest)}
\date{\emph{In memory of Professor Lajos Tam\'assy}}
\footnotetext{2010 {\em Mathematics Subject Classification: 20N05}.}
\footnotetext{{\em Key words and phrases:} differentiable loops, tangent prolongation of smooth multiplications, abelian extensions, weak inverse and weak associative properties}
\footnotetext{This paper was supported by the National Research, Development and Innovation Office (NKFIH) Grant No. K132951, and by the EFOP-3.6.1-16-2016-00022 project of the European Union, co-financed by the European Social Fund.}
\maketitle
\begin{abstract} The aim of our paper is to generalize the tangent prolongation of Lie groups to non-associative multiplications and to examine how the weak associative and weak inverse properties are transferred to the multiplication defined on the tangent bundle. We obtain that the tangent prolongation of a $\mathcal{C}^r$-differentiable loop ($r\geq 1$) is a $\mathcal{C}^{r-1}$-differentiable loop that has the classical weak inverse and weak associative properties of the initial loop. 
\end{abstract}

\section{Introduction}

	The tangent space of a differentiable manifold $M$ at $\xi\in M$ is $\mathrm{T}_\xi (M)$ and the tangent bundle is $\mathrm{T}(M)$. The tangent map of a differentiable map $\varphi : M\to M$ at $\xi\in M$ will be denoted by  $\mathrm{d}_\xi\varphi:\mathrm{T}_\xi (M)\to \mathrm{T}_{\varphi(\xi)}(M)$. For a Lie group $G$ with Lie algebra $\mathfrak{g}$ any element $(\xi,X)\in\mathrm{T}(G)$ can be identified with the element $(\xi,\mathrm{d}_\xi\lambda^{-1}_\xi X)$ of the Cartesian product $G\times \mathfrak{g}$, where $\lambda_\xi:G\to G$ denotes the left translation on $G$.  It is well-known that the Lie group structure of $G$ has a natural prolongation to tangent bundle $\mathrm{T}(G)$ such that the corresponding multiplication on $G\times \mathfrak{g}$ is expressed by    
\begin{equation}\label{lietanbun} (\xi,x)\cdot(\eta,y) = \big(\xi\eta,\mathrm{d}_{\xi\eta}\lambda^{-1}_{\xi\eta}\frac{d}{dt}\big|_{t=0}(\xi\exp tx\cdot\eta\exp ty)\big) =  \big(\xi\eta,\mathrm{d}_{\epsilon}(\lambda^{-1}_{\eta}\rho_{\eta})x + y\big),\end{equation}
where $\xi,\eta\in G$ and $x,y\in\mathfrak{g}$. Clearly, the tangent prolongation structure on 
$G\times \mathfrak{g}$ is a semi-direct product $G\ltimes\mathfrak{g}^+$, where $\mathfrak{g}^+$ is the additive group of $\mathfrak{g}$, cf. \cite{YI}, \cite{KMS}, \cite{Sau}. \par
The aim of our paper is to generalize the tangent prolongation (\ref{lietanbun}) to non-associative multiplications and to show that the weak associative and weak inverse properties are preserved by the tangent prolongation. We prove that the tangent prolongation of a $\mathcal{C}^r$-differentiable loop ($r\geq 1$) is a $\mathcal{C}^{r-1}$-differentiable loop that has the same classical weak inverse and weak associative properties as the initial loop. It is worth noting that, unlike the associative case, the multiplications of differentiable loops are not necessarily analytic, (cf. the loop constructions in Part 2 of \cite{nst}). Our research is a reflection on the subject raised by J. Grabowski  in his conference lecture \emph{Tangent and cotangent loopoids} \cite{Gra}.\par
We examine in \S 3 the two-sided inverse, left and right inverse, monoassociative, left and right alternative, flexible, left and right Bol properties of general abelian linear extensions. The loops in this class of extensions (cf. \cite{StVo2}, \cite{Fi_Na}) are natural abstract generalizations of the tangent prolongation of differentiable loops. \S 4 is devoted to the discussion of the differentiability properties of tangent prolongations. In \S 5 we investigate such abelian linear extensions, which are determined by a homomorphism of the inner mapping group into the linear group acting on the tangent space at the identity element. We apply our results to the tangent prolongation of $\mathcal{C}^r$-differentiable loops, where the above homomorphism is determined by the tangent map of inner mappings. At the end in \S 6 we obtain our main result on the tangent prolongation of $\mathcal{C}^r$-differentiable loops and on their classical weak inverse and weak associative properties.

\section{Preliminaries}

There are two usual formal definitions of quasigroups: as a set with a binary multiplication, or the other, as a set with three binary operations: multiplication, left division and right division. We will study, among others, the differentiability properties of each operation; therefore, it is worth using the definition with three operations.\par
A \emph{quasigroup} $(Q,\cdot,\backslash,/)$ is a set  equipped with three binary operations: \emph{multiplication} $\cdot$,  \emph{left division} $\backslash$ and \emph{ right division} $/$ satisfying the identities
\[y = x\cdot(x\backslash y),\quad y = x\backslash(x\cdot y),\quad y = (y /x) \cdot x,\quad y = (y\cdot x)/ x,\quad x,y\in Q.\]
A \emph{loop} $(L,\cdot,\backslash,/)$ is a quasigroup with an identity element $e\in L$, ($e\cdot x = x\cdot e = x$ for each $x\in L$). The \emph{left translations} $\lambda_x : y \mapsto x\cdot y$,  and the \emph{right translations} $\rho_x : y \mapsto y\cdot x$ are bijective maps and $x\backslash y = \lambda_x^{-1}y$, $x/y = \rho_y^{-1}x$, $x,y\in L$. The \emph{opposite loop} of $(L,\cdot,\backslash,/)$ is the loop $(L,\star,\backslash^\star,/^\star)$ defined by the operations $x\star y = y\cdot x$, $x\backslash^\star y = y/x$, $y/^\star x = x\backslash y$ on the same underlying set.  \par 
The automorphism group of a loop $L$ is denoted by $\mathrm{Aut}(L)$.
The group $\mathrm{Mlt}(L)$ generated by all left and right translations of the loop $L$ is called the \emph{multiplication group} of $L$. The \emph{inner mapping group} $\mathrm{Inn}(L) = \{\varphi\in\mathrm{Mlt}(L)\,;\,\varphi(e)= e\}$ of $L$ is the subgroup of $\mathrm{Mlt}(L)$ consisting of all bijections in $\mathrm{Mlt}(L)$ fixing the identity element $e$.\par
We will reduce the use of parentheses on the loop $L$ by the following convention: juxtaposition will denote multiplication, the division operations are less binding than juxtaposition, and the multiplication is less binding than the divisions. For instance the expression $xy / u \cdot v \backslash w$ is a short form of $((x\cdot y)/ u)\cdot(v\backslash w)$.\par 
We say that a loop $L$ has a \emph{weak associative property} if it satisfies one of the following properties:\\[1ex]
$\begin{array}{l l l}
\text{\emph{monoassociative}}: & x \cdot x^2=x^2 \cdot x  &\text{for all} \quad x \in L, \\
\text{\emph{left alternative}}: & x \cdot xy=x^2 \cdot y & \text{for all} \quad x,y \in L,  \\ 
\text{\emph{right alternative}}: & yx \cdot x=y \cdot x^2 & \text{for all} \quad x,y \in L, \\ 
\text{\emph{flexible}} & x \cdot yx=xy \cdot x & \text{for all} \quad x,y \in L,  \\
\text{\emph{left Bol loop}}: & (x \cdot yx)z=x(y \cdot xz) & \text{for all} \quad x,y,z \in L, \\
\text{\emph{right Bol loop}}: & z(xy \cdot x)=(zx \cdot y)x & \text{for all} \quad x,y,z \in L.
\end{array}$\vspace{10pt}\par
The element $e/x$, respectively, 
$x \backslash e$ is the \emph{left inverse}, respectively, the \emph{right inverse} of $x\in L$. If the left and the right inverses of $x$ coincide then $x^{-1}= e/x = x \backslash e$ is the \emph{inverse} of $x\in L$. If any element 
$x\in L$ has an inverse $x^{-1}$ then we say that $L$ has \emph{two-sided inverse property}.  
$L$ satisfies the \emph{left}, respectively, the \emph{right inverse property} 
if there exists a bijection $\iota:L\to L$, such that $\iota(x)\cdot xy = y$, respectively, 
$yx\cdot \iota(x)= y$ holds for all $x, y\in L$.  It is well known that in loops with left or right inverse property all elements have inverses (cf. \cite{Pfl}, I.4.2 Theorem), hence $\iota(x) = x^{-1}$.\par 
We say that a loop $L$ has a \emph{weak inverse property} if it has one of the following properties: \\[1ex]
$\begin{array}{l l l}
\text{\emph{two-sided inverse}}: & x^{-1}= e/x = x \backslash e  &\text{for all} \quad x \in L, \\
\text{\emph{left inverse}}: & x^{-1} \cdot xy=y & \text{for all} \quad x,y \in L,  \\ 
\text{\emph{right inverse}}: & yx \cdot x^{-1}=y  & \text{for all} \quad x,y \in L.
\end{array}$\vspace{10pt}\par
A loop $L$ defined on a differentiable manifold is called \emph{$\mathcal{C}^r$-differentiable loop}, 
$0<r\in\mathbb N$ or $r =\infty$, if the multiplication $(x,y)\mapsto x\cdot y$, the left division 
$(x,y)\mapsto x\backslash y$ and the right division $(x,y)\mapsto x/y$ are $\mathcal{C}^r$-differentiable 
$L\times L\to L$ maps.

\section{Linear abelian extensions}

Before the introduction of a natural loop-multiplication on the tangent bundle of  differentiable loops we consider a more general class of abstract non-associative extensions. If $(A, +)$ is an abelian group, $\varphi,\psi$ are automorphisms of $A$ and $c\in A$ is a constant, then the binary operations 
\begin{equation}\label{tquasi} x\cdot y = \varphi(x) + \psi(y) + c,\quad x\backslash y = \psi^{-1}(y - \varphi(x) - c),\quad y/x = \varphi^{-1}(y - \psi(x) - c) \end{equation}
on $A$ determine a quasigroup $(A,\cdot,\backslash,/)$, called \emph{T-quasigroup} or \emph{central quasigroup} over the abelian group $A$, investigated intensively in the last half century, (cf. e.g. \cite{KeNe1}, \cite{KeNe2}, \cite{Vic}, \cite{JDH}).
The study of \emph{non-associative abelian extensions} of abelian groups by loops, such that the quasigroups, induced between cosets of the kernel subgroups, are T-quasigroups is initiated by D. Stanovsk\'y and  P. Vojt\v{e}chovsk\'y in \cite{StVo2}. We investigated in a recent paper \cite{Fi_Na} the two-sided inverse, left-inverse, right-inverse and
inverse properties of a class of abelian extensions, called  \emph{abelian linear extensions}. \par
Let $A = (A,+)$ be an abelian group and $L = (L,\cdot, /, \backslash )$ a loop with the  identity element $\epsilon\in L$. A pair 
$(P, Q)$ is called a \emph{loop cocycle} if $P, Q$ are mappings $L\times L\to\mathrm{Aut}(A)$
satisfying  $P(\xi,\epsilon)=\mathrm{Id}=Q(\epsilon,\eta)$ for every $\xi,\eta \in L$. 
\begin{definition} \label{ext2} If $(P, Q)$ is a loop cocycle then the binary operations on $L \times A$ 
	\begin{equation} \label{mult}\begin{split} &(\xi,x) \cdot (\eta,y)=(\xi\eta,P(\xi,\eta) x + Q(\xi,\eta) y), \\
	&(\xi,x)\backslash(\eta,y) = \left(\xi\backslash\eta,Q(\xi,\xi\backslash\eta)^{-1}( y - P(\xi,\xi\backslash\eta)x)\right),\\  
	&(\eta,y)/(\xi,x) = \left(\eta/\xi,P(\eta/\xi,\xi)^{-1}(y - Q(\eta/\xi,\xi) x)\right) 
	\end{split}\end{equation}
	define the \emph{linear abelian extension} $F(P,Q)$ of the group $A$ by the loop $L$. The linear abelian extension $F(P,Q)$ is a loop with the identity element $(\epsilon,0)$.\end{definition}
Definition \ref{ext2} implies that if a linear abelian extension $F(P,Q)$ has one of the weak inverse or weak associative properties then the loop $L$ has necessarily the corresponding property. The left and right inverses in $F(P,Q)$ are 
\[\begin{split}(\xi,x)\backslash(\epsilon,0) = (\xi\backslash\epsilon,-Q(\xi,\xi\backslash\epsilon)^{-1} P(\xi,\xi\backslash\epsilon)x),\\ (\epsilon,0)/(\xi,x) = (\epsilon/\xi,-P(\epsilon/\xi,\xi)^{-1} 
Q(\epsilon/\xi,\xi) x),\end{split}\]
hence we obtain, (cf. \cite{Fi_Na}):
\begin{prop}\label{leftandrightinverses}  
	If the loop $L$ has one of the weak inverse or weak associative properties, the extension $F(P,Q)$ has the corresponding property if and only if the loop cocycle $(P,Q)$ fulfills the identities given in the following list for each property:
	\begin{enumerate}
		\item[\emph{(A)}] two-sided inverse:
		\[  P(\xi,\xi^{-1}) = Q(\xi,\xi^{-1})P(\xi^{-1},\xi)^{-1}Q(\xi^{-1},\xi),  
		\]
		\item[\emph{(B)}] left inverse:
		\[ Q(\xi^{-1},\xi\eta) = Q(\xi,\eta)^{-1}, \quad P(\xi^{-1},\xi\eta) = Q(\xi,\eta)^{-1}P(\xi,\eta)Q(\xi^{-1},\xi)^{-1}P(\xi^{-1},\xi), \]
		\item[\emph{(C)}] right inverse:
		\[ P(\xi\eta,\eta^{-1})=P(\xi,\eta)^{-1}, \quad Q(\xi\eta,\eta^{-1})=P(\xi,\eta)^{-1}Q(\xi,\eta)P(\eta,\eta^{-1})^{-1}Q(\eta,\eta^{-1}), \]
		\item[\emph{(D)}] monoassociative:
		\[ P(\xi,\xi^2)+ Q(\xi,\xi^2)\left(P(\xi,\xi) + Q(\xi,\xi)\right)=P(\xi^2,\xi)(P(\xi,\xi)+Q(\xi,\xi)) + Q(\xi^2,\xi),
		\]
		\item[\emph{(E)}] left alternative:     
		\[\begin{split}  Q(\xi,\xi\eta)Q(\xi,\eta)&=Q(\xi^2,\eta), \\ P(\xi,\xi\eta)+ Q(\xi,\xi\eta) P(\xi,\eta)&=P(\xi^2,\eta)(P(\xi,\xi)+Q(\xi,\xi)), \end{split} \]
		\item[\emph{(F)}] right alternative: 
		\[\begin{split} P(\eta\xi,\xi)P(\eta,\xi)&=P(\eta,\xi^2), \\ P(\eta\xi,\xi) Q(\eta,\xi)+ Q(\eta\xi,\xi)&=Q(\eta,\xi^2)(P(\xi,\xi)+Q(\xi,\xi)),  \end{split} \]
		\item[\emph{(G)}] flexible: 
		\[\begin{split} 	Q(\xi,\eta\xi)P(\eta,\xi)&=P(\xi\eta,\xi)Q(\xi,\eta),\\
		P(\xi,\eta\xi)+ Q(\xi,\eta\xi) Q(\eta,\xi)&=P(\xi\eta,\xi)P(\xi,\eta)+Q(\xi\eta,\xi), 
		\end{split} \] 
		\item[\emph{(H)}] left Bol: 
		\[\begin{split} Q(\xi,\eta \cdot \xi\zeta)Q(\eta,\xi\zeta)&Q(\xi,\zeta)=Q(\xi \cdot \eta\xi,\zeta),\\ Q(\xi,\eta\cdot\xi\zeta)P(\eta,\xi\zeta)= &P(\xi \cdot \eta\xi,\zeta)Q(\xi,\eta\xi)P(\eta,\xi),\\P(\xi,\eta\cdot \xi\zeta)+ Q(\xi,\eta\cdot &\xi\zeta)Q(\eta,\xi\zeta)P(\xi,\zeta)=\\
		=P(\xi \cdot\eta\xi,\zeta)\big(P(\xi,\eta&\xi)+ Q(\xi,\eta\xi) Q(\eta,\xi)\big), \end{split}\] 
		\item[\emph{(J)}] right Bol: 
		\[\begin{split} P(\zeta,\xi\eta\cdot\xi)=P(\zeta\xi\cdot\eta,\xi)&P(\zeta\xi,\eta)P(\zeta,\xi), 
		\\ Q(\zeta,\xi\eta\cdot \xi)P(\xi\eta,\xi)Q(\xi,\eta)=&P(\zeta\xi\cdot\eta,\xi)Q(\zeta\xi,\eta), \\ Q(\zeta,\xi\eta\cdot\xi)\big(P(\xi\eta,\xi)P(\xi,&\eta)+ Q(\xi\eta,\xi)\big)=\\= P(\zeta\xi\cdot\eta,\xi) P(\zeta\xi,\eta)Q(&\zeta,\xi)+Q(\zeta\xi\cdot \eta,\xi). \end{split}\]
	\end{enumerate}
\end{prop}
\begin{proof}
	Assertions (A), (B), (C) with respect to the weak inverse properties of the extension $F(P,Q)$ are proved by Propositions 3, 5 and 7 in \cite{Fi_Na}.\par
	Assertion (D) follows from the equations 
	\begin{equation} \label{mono2}\begin{split} &(\xi,x)\cdot(\xi,x)^2 = (\xi \cdot \xi^2,P(\xi,\xi^2)x + Q(\xi,\xi^2)(P(\xi,\xi)x + Q(\xi,\xi)x)= \\
	&= (\xi,x)^2\cdot(\xi,x) = (\xi^2 \cdot \xi,P(\xi^2,\xi)(P(\xi,\xi)x + Q(\xi,\xi)x )+ Q(\xi^2,\xi)x)\end{split}
	\end{equation} 
	for all $\xi\in L$ and $x\in A$.\par
	The components of the left alternative identity of
	$F(P,Q)$ give the identity     
	\[ \begin{split}  P(\xi,\xi\eta)x + Q(\xi,\xi\eta)(P(\xi,\eta)x + Q(\xi,\eta)y)= \\ = P(\xi^2,\eta)(P(\xi,\xi)x + Q(\xi,\xi)x )+ Q(\xi^2,\eta)y \end{split}
	\] 
	for all $\xi,\eta \in L$ and $x,y \in A$. Putting $y = 0$, respectively, $x = 0$, we get the  identities (E).\par
	Considering the opposite loop of $L$ we obtain (F). 
	Similar computations give the condition (G) of flexible property.\par
	The components of the left Bol identity in $F(P,Q)$ give the equation    
	\begin{equation} \label{eq24}\begin{split} P(\xi,\eta \cdot \xi\zeta)x + Q(\xi,\eta\cdot\xi\zeta)[P(\eta,\xi\zeta)y + Q(\eta,\xi\zeta)(P(\xi,\zeta) x +Q(\xi,\zeta) z )]= \\ 
	P(\xi \cdot (\eta\xi),\zeta)[P(\xi,\eta\xi)x + Q(\xi,\eta\xi)(P(\eta,\xi) y + Q(\eta,\xi) x)]+ Q(\xi \cdot\eta\xi,\zeta)z\end{split}
	\end{equation} 
	for all $\xi,\eta,\zeta \in L$ and $x,y,z \in A$. Putting $x=y=0$ into the identity (\ref{eq24}), we obtain the first identity in assertion (H), the substitutions $x=z=0$ and $y=z=0$ give the second and the third identities in (H).\par
	We obtain condition (J) from (H) using the opposite loop of $L$.
\end{proof}

	\section{Tangent prolongation}

We extend the construction (\ref{lietanbun}) of the tangent prolongation of Lie groups to $\mathcal{C}^r$-differentiable loops $L$, $r\geq 1$. Similarly to Lie groups, the map $\mathscr{T}:(\xi,x)\mapsto \mathrm{d}_\epsilon\lambda_\xi x$ for any 
$(\xi,x)\in L \times \mathrm{T}_\epsilon(L)$ gives an identification $\mathscr{T}:L \times \mathrm{T}_\epsilon(L)\to \mathrm{T}(L)$ of the Cartesian product manifold $L \times \mathrm{T}_\epsilon(L)$ with the tangent bundle $\mathrm{T}(L)$. 
\begin{definition} Let $\alpha(t)$, $\beta(t)$ be differentiable curves in the $\mathcal{C}^r$-differentiable loop $L$ ($r\geq 1$) with initial data $\alpha(0) =\beta(0) = \epsilon$ and  
	$\alpha'(0) = x$, $\beta'(0) = y$, $x,y \in \mathrm{T}_\epsilon(L)$. We define the multiplication of the \emph{tangent prolongation} $\mathscr{T}(L \times \mathrm{T}_\epsilon(L))$ of  $L$ on $L \times \mathrm{T}_\epsilon(L)$ by 
	\[ (\xi,x)\cdot(\eta,y) = \big(\xi\eta,\mathrm{d}_{\xi\eta}\lambda^{-1}_{\xi\eta}\frac{d}{dt}\big|_{t=0}(\xi\alpha(t)\cdot\eta\beta(t))\big) =
	(\xi\eta,\mathrm{d}_\epsilon\lambda_{\xi\eta}^{-1} \rho_\eta \lambda_\xi x+ \mathrm{d}_\epsilon \lambda_{\xi\eta}^{-1} \lambda_\xi \lambda_\eta y).\] 
	The identity element of $\mathscr{T}(L \times \mathrm{T}_\epsilon(L))$ is $(\epsilon,0)$.
\end{definition}
Clearly, the maps $\lambda_{\xi\eta}^{-1} \rho_\eta \lambda_\xi$ and $\lambda_{\xi\eta}^{-1} \lambda_\xi \lambda_\eta$ are inner mappings of the loop $L$. They are differentiable maps and the assignment $\phi\mapsto\mathrm{d}_\epsilon\phi$ is a homomorphism $\mathrm{d}_\epsilon: \mathrm{Inn}(L) \to\mathrm{Aut}(\mathrm{T}_\epsilon(L))$.  Naturally, the map $\varphi\mapsto \mathrm{d}_\xi\varphi$ with arbitrary $\xi\in L$ acts on all elements $\varphi\in\mathrm{Mlt}(L)$ of the multiplication group $\mathrm{Mlt}(L)$ of $L$.
\begin{lem} \label{cocyclepro}
	Let $L$ be a $\mathcal{C}^r$-differentiable loop and $\mathrm{T}_\epsilon(L)$ its tangent vector space. The tangent prolongation $\mathscr{T}(L \times \mathrm{T}_\epsilon(L))$ of $L$ is the linear abelian extension $F(P,Q)$ of the abelian group $\mathrm{T}_\epsilon(L)$ by $L$ determined by the loop cocycle $(P,Q)$ of $\mathscr{T}(L \times \mathrm{T}_\epsilon(L))$ given by the maps 
	\[ P(\xi,\eta):=\mathrm{d}_\epsilon(\lambda_{\xi\eta}^{-1} \rho_\eta \lambda_\xi), \quad  Q(\xi,\eta):=\mathrm{d}_\epsilon(\lambda_{\xi\eta}^{-1} \lambda_\xi \lambda_\eta). \]
\end{lem}
\begin{prop}\label{tangentdiff} The tangent prolongation $\mathscr{T}(L \times \mathrm{T}_\epsilon(L))$ of a $\mathcal{C}^r$-differentiable loop $L$ ($r\geq 1$) is a $\mathcal{C}^{r-1}$-differentiable loop.
\end{prop}
\begin{proof} Since the multiplication, the left and right divisions of $L$ are $\mathcal{C}^r$-differentiable maps, and $\lambda_\sigma\tau = \sigma\tau$, $\rho_\sigma\tau =  \tau\sigma$, $\lambda^{-1}_\sigma\tau = \sigma\backslash\tau$, $\rho^{-1}_\sigma\tau =  \tau/\sigma$, the left and right translations and their inverses are $\mathcal{C}^r$-differentiable maps, too. According to Lemma \ref{cocyclepro} the loop cocycle $(P,Q)$ of $\mathscr{T}(L \times \mathrm{T}_\epsilon(L))$ is expressed by the first derivative of the products of left and right translations and of their inverses. Hence $P(\xi,\eta)$, $Q(\xi,\eta)$, and the multiplication of $\mathscr{T}(L \times \mathrm{T}_\epsilon(L))$ determined by (\ref{mult}) are  $\mathcal{C}^{r-1}$-differentiable. According to (\ref{mult}) the left and right divisions of the tangent prolongation $\mathscr{T}(L \times \mathrm{T}_\epsilon(L))$ are $\mathcal{C}^{r-1}$-differentiable, since the maps 
	\[ P(\xi,\eta)^{-1}:=\mathrm{d}_\epsilon(\lambda_\xi^{-1} \rho_\eta^{-1} \lambda_{\xi\eta}), \quad  Q(\xi,\eta)^{-1}:=\mathrm{d}_\epsilon(\lambda_\eta^{-1}\lambda_\xi^{-1} \lambda_{\xi\eta}). \]
	are $\mathcal{C}^{r-1}$-differentiable. It follows that the tangent prolongation $\mathscr{T}(L \times \mathrm{T}_\epsilon(L))$ is a $\mathcal{C}^{r-1}$-differentiable loop.
\end{proof}

	\section{Tangent-like extension}

Now, we treat the properties of the tangent prolongation in a more general abstract setting, considering  an arbitrary homomorphism $\Phi: \mathrm{Inn}(L) \to\mathrm{Aut}(A)$ instead of the assignment $\mathrm{d}_\epsilon: \mathrm{Inn}(L) \to\mathrm{Aut}(\mathrm{T}_\epsilon(L))$, furthermore we apply the obtained conditions to the tangent prolongation of $L$. 
\begin{definition}\label{dphi}
	Let $L$ be a loop, $A$ an abelian group and $\Phi: \mathrm{Inn}(L) \to\mathrm{Aut}(A)$ a homomorphism. The linear abelian extension $F(P,Q)$ defined by the loop cocycle $(P,Q)$  
	\begin{equation} \label{mapPQ} P(\xi,\eta):=\Phi(\lambda_{\xi\eta}^{-1} \rho_\eta \lambda_\xi), \quad  Q(\xi,\eta):=\Phi(\lambda_{\xi\eta}^{-1} \lambda_\xi \lambda_\eta). \end{equation} 
	is called a \emph{tangent-like extension} of the group $A$ by the loop $L$ and will be denoted by $\Phi(L,A)$.
\end{definition}
\begin{prop}\label{tangent}
	If the loop $L$ has one of the weak inverse or weak associative properties, the tangent-like extension $\Phi(L,A)$, respectively, the tangent prolongation $\mathscr{T}(L \times \mathrm{T}_\epsilon(L))$ has the corresponding  property if and only if the loop cocycle fulfills the identities given in the following list for each property:
	\begin{enumerate}
		\item[\emph{(A)}] two-sided inverse:
		\begin{enumerate}
			\item[\emph{(i)}] $\Phi(L,A):\quad\quad\quad\quad\Phi(\rho_{\xi^{-1}} \lambda_\xi) = \Phi(\lambda_\xi \rho_\xi^{-1}\lambda_{\xi^{-1}}\lambda_\xi),$
			\item[\emph{(ii)}] $\mathscr{T}(L \times \mathrm{T}_\epsilon(L)):\quad\; \mathrm{d}_\xi\rho_{\xi^{-1}}  = \mathrm{d}_\xi\lambda_\xi \rho_\xi^{-1}\lambda_{\xi^{-1}}$,	\end{enumerate} 
		\item[\emph{(B)}] left inverse:
		\begin{enumerate}
			\item[\emph{(i)}] $\Phi(L,A):\quad\quad\quad\quad\Phi(\lambda_{\eta}^{-1} \rho_{\xi\eta} \lambda_\xi^{-1}) = \Phi(\lambda_\eta^{-1}\lambda_\xi^{-1} \rho_\eta \lambda_\xi\rho_\xi\lambda_\xi^{-1})$,
			\item[\emph{(ii)}] $\mathscr{T}(L \times \mathrm{T}_\epsilon(L)):\quad\; \mathrm{d}_{\xi^{-1}} \lambda_\xi \rho_{\xi\eta} = \mathrm{d}_{\xi^{-1}} \rho_\eta \lambda_\xi\rho_\xi$,	\end{enumerate}
		\item[\emph{(C)}] right inverse:
		\begin{enumerate}
			\item[\emph{(i)}] $\Phi(L,A):\quad\quad\quad\quad  \Phi(\lambda_\xi^{-1} \lambda_{\xi\eta} \lambda_{\eta^{-1}})= \Phi(\lambda_\xi^{-1} \rho_\eta^{-1}\lambda_\xi\rho_\eta\lambda_\eta \lambda_{\eta^{-1}}),$
			\item[\emph{(ii)}] $\mathscr{T}(L \times \mathrm{T}_\epsilon(L)):\quad\; \mathrm{d}_{\eta^{-1}} \rho_\eta\lambda_{\xi\eta} = \mathrm{d}_{\eta^{-1}} \lambda_\xi\rho_\eta\lambda_\eta$,	\end{enumerate}
		\item[\emph{(D)}] monoassociative:
		\begin{enumerate}
			\item[\emph{(i)}] $\Phi(L,A):\quad\quad\quad\quad \Phi(\lambda_{\xi^3}^{-1} \rho_{\xi^2} \lambda_\xi)+ \Phi(\lambda_{\xi^3}^{-1} \lambda_\xi \rho_\xi \lambda_\xi) + \Phi(\lambda_{\xi^3}^{-1} \lambda_\xi^3) =$\\   
			\hspace*{32mm}$ = \Phi(\lambda_{\xi^3}^{-1} \rho_\xi \rho_\xi\lambda_\xi)+ \Phi(\lambda_{\xi^3}^{-1} \rho_\xi\lambda_\xi^2) + \Phi(\lambda_{\xi^3}^{-1} \lambda_{\xi^2} \lambda_\xi),$                                                                          
			\item[\emph{(ii)}] $\mathscr{T}(L \times \mathrm{T}_\epsilon(L)):\quad\; \mathrm{d}_\xi \rho_{\xi^2} + \mathrm{d}_\xi \lambda_\xi \rho_\xi  + \mathrm{d}_\xi  \lambda_\xi^2 = \mathrm{d}_\xi \rho_\xi^2 + \mathrm{d}_\xi  \rho_\xi\lambda_\xi + \mathrm{d}_\xi \lambda_{\xi^2}$,	\end{enumerate}
		\item[\emph{(E)}] left alternative: 
		\begin{enumerate}
			\item[\emph{(i)}] $\Phi(L,A):\quad\quad\quad\quad \Phi(\lambda_{\xi\cdot\xi\eta}^{-1} \rho_{\xi\eta} \lambda_\xi) + \Phi(\lambda_{\xi\cdot\xi\eta}^{-1} \lambda_\xi \rho_\eta \lambda_\xi)=$\\   
			\hspace*{32mm}$ = \Phi(\lambda_{\xi^2\eta}^{-1} \rho_\eta \rho_\xi \lambda_\xi)+ \Phi(\lambda_{\xi^2\eta}^{-1} \rho_\eta \lambda_\xi^2),$ 
			\item[\emph{(ii)}] $\mathscr{T}(L \times \mathrm{T}_\epsilon(L)):\quad\; \mathrm{d}_\xi \rho_{\xi\eta} + \mathrm{d}_\xi \lambda_\xi \rho_\eta = \mathrm{d}_\xi  \rho_\eta \rho_\xi + \mathrm{d}_\xi \rho_\eta \lambda_\xi$,	\end{enumerate}
		\item[\emph{(F)}] right alternative: 		\begin{enumerate}
			\item[\emph{(i)}] $\Phi(L,A):\quad\quad\quad\quad \Phi(\lambda_{\eta\xi\cdot\xi}^{-1} \rho_\xi \lambda_\eta\lambda_\xi) + \Phi(\lambda_{\eta\xi\cdot\xi}^{-1} \lambda_{\eta\xi} \lambda_\xi) =$\\   
			\hspace*{32mm}$ = \Phi(\lambda_{\eta\xi^2}^{-1} \lambda_\eta\rho_\xi \lambda_\xi) + \Phi(\lambda_{\eta\xi^2}^{-1} \lambda_\eta\lambda_\xi^2),$ 
			\item[\emph{(ii)}] $\mathscr{T}(L \times \mathrm{T}_\epsilon(L)):\quad\; \mathrm{d}_\xi \rho_\xi \lambda_\eta + \mathrm{d}_\xi \lambda_{\eta\xi} = \mathrm{d}_\xi \lambda_\eta\rho_\xi + \mathrm{d}_\xi \lambda_\eta\lambda_\xi $,	\end{enumerate}
		\item[\emph{(G)}] flexible: 
		\begin{enumerate}
			\item[\emph{(i)}] $\Phi(L,A):\quad\quad\quad\quad \Phi(\lambda_{\xi\cdot\eta\xi}^{-1} \rho_{\eta\xi} \lambda_\xi)+ \Phi(\lambda_{\xi\cdot\eta\xi}^{-1} \lambda_\xi\lambda_\eta\lambda_\xi) =$\\   
			\hspace*{32mm}$ = \Phi(\lambda_{\xi\eta\cdot\xi}^{-1} \rho_\xi\rho_\eta \lambda_\xi)+ \Phi(\lambda_{\xi\eta\cdot\xi}^{-1} \lambda_{\xi\eta} \lambda_\xi),$ 
			\item[\emph{(ii)}] $\mathscr{T}(L \times \mathrm{T}_\epsilon(L)):\quad\;\mathrm{d}_\xi \rho_{\eta\xi} + \mathrm{d}_\xi \lambda_\xi\lambda_\eta = \mathrm{d}_\xi \rho_\xi\rho_\eta + \mathrm{d}_\xi \lambda_{\xi\eta}$,	\end{enumerate}
		\item[\emph{(H)}] left Bol: 
		\begin{enumerate}
			\item[\emph{(i)}] $\Phi(L,A):\quad\quad\quad\quad \Phi(\lambda_{\xi(\eta\cdot \xi\zeta)}^{-1} \rho_{\eta\cdot \xi\zeta} \lambda_\xi)+ \Phi(\lambda_{\xi(\eta\cdot \xi\zeta)}^{-1} \lambda_\xi \lambda_\eta\rho_\zeta \lambda_\xi)=$\\   
			\hspace*{32mm}$= \Phi(\lambda_{(\xi \cdot \eta\xi)\zeta}^{-1} \rho_\zeta\rho_{\eta\xi} \lambda_\xi)+\Phi(\lambda_{(\xi \cdot \eta\xi)\zeta}^{-1} \rho_\zeta \lambda_\xi \lambda_\eta\lambda_\xi),$ 
			\item[\emph{(ii)}] $\mathscr{T}(L \times \mathrm{T}_\epsilon(L)):\quad\;\mathrm{d}_\xi \rho_{\eta\cdot \xi\zeta} + \mathrm{d}_\xi  \lambda_\xi \lambda_\eta\rho_\zeta =
			\mathrm{d}_\xi \rho_\zeta\rho_{\eta\xi} + \mathrm{d}_\xi \rho_\zeta \lambda_\xi \lambda_\eta$,	\end{enumerate}
		\item[\emph{(J)}] right Bol: 
		\begin{enumerate}
			\item[\emph{(i)}] $\Phi(L,A):\quad\quad\quad\quad \Phi(\lambda_{\zeta(\xi\eta\cdot\xi)}^{-1} \lambda_\zeta\rho_\xi\rho_\eta \lambda_\xi)+ \Phi(\lambda_{\zeta(\xi\eta\cdot\xi)}^{-1} \lambda_\zeta\lambda_{\xi\eta}\lambda_\xi) = $\\   
			\hspace*{32mm}$ =  \Phi(\lambda_{(\zeta\xi\cdot\eta)\xi}^{-1} \rho_\xi\rho_\eta\lambda_\zeta \lambda_\xi)+ \Phi(\lambda_{(\zeta\xi\cdot\eta)\xi}^{-1} \lambda_{\zeta\xi\cdot\eta} \lambda_\xi)$, 
			\item[\emph{(ii)}] $\mathscr{T}(L \times \mathrm{T}_\epsilon(L)):\quad\;\mathrm{d}_\xi \lambda_\zeta\rho_\xi\rho_\eta + \mathrm{d}_\xi \lambda_\zeta\lambda_{\xi\eta} =  \mathrm{d}_\xi \rho_\xi\rho_\eta\lambda_\zeta + \mathrm{d}_\xi \lambda_{\zeta\xi\cdot\eta}$. 	\end{enumerate}
	\end{enumerate}
\end{prop}
\begin{proof} Substituting the expression (\ref{mapPQ}) of the loop cocycle of the tangent-like extension $\Phi(L,A)$ into the first identities of (B), (C), (E), (F), (G), respectively, into the first two identities of  (H), (J) in Proposition \ref{leftandrightinverses}, we get relationships, which are equivalent to the identities that characterize the corresponding weak inverse or weak associative property of $L$. The replacement of (\ref{mapPQ}) into the further identities results in non-trivial conditions for each property of the tangent-like extension $\Phi(L,A)$, which are  listed in the items (i) of Proposition \ref{tangent}. \par
The inner mapping group $\mathrm{Inn}(L)$ is a subgroup of the group generated by all left and right translations.
The map $\Phi: \mathrm{Inn}(L) \to\mathrm{Aut}(A)$ is defined only on $\mathrm{Inn}(L)$, but the assignment $\phi\mapsto\mathrm{d}_\xi\phi$, $\xi\in L$ is defined on all left and right translations of the $\mathcal{C}^r$-differentiable loop $L$. Hence we can simplify the equations obtained by the application of the conditions for the tangent-like extension to the tangent prolongation. We list for each property the corresponding necessary and sufficient conditions of the tangent prolongation $\mathscr{T}(L \times \mathrm{T}_\epsilon(L))$ in the items (ii) of Proposition \ref{tangent}.
\end{proof}

\section{Properties of the tangent prolongation}

\begin{theor} \label{mainres} The tangent prolongation $\mathscr{T}(L \times \mathrm{T}_\epsilon(L))$ of a $\mathcal{C}^r$-differentiable loop $L$ ($r\geq 1$) is a $\mathcal{C}^{r-1}$-differentiable loop that has a weak inverse or weak associative property if and only if $L$ has the corresponding property.
\end{theor} 
\begin{proof} According to Proposition \ref{tangentdiff} the tangent prolongation $\mathscr{T}(L \times \mathrm{T}_\epsilon(L))$ of $L$ is $\mathcal{C}^{r-1}$-differentiable. It follows from Lemma \ref{cocyclepro} and Proposition \ref{leftandrightinverses}, that a weak inverse or weak associative property of the tangent prolongation $\mathscr{T}(L \times T_\epsilon(L))$ implies the corresponding property of $L$. Let us consider a differentiable curve $\xi(t)$ in $L$ with initial values $\xi(0) = \xi$ and $\frac{d\xi}{dt}(0) = \dot\xi$ defined on an open interval $I\subset\mathbb R$ containing $0\in I$. \\
	If all elements of $L$ have two-sided inverses then the derivation of the identity $\xi(t)^{-1}\cdot\xi(t) = \epsilon$ at $t = 0$ gives 
	\[\mathrm{d}_{\xi^{-1}}\rho_\xi (\dot\xi^{-1}) + \mathrm{d}_{\xi}\lambda_{\xi^{-1}}(\dot\xi) = 0,\] where $\dot\xi^{-1}: = \frac{d\xi^{-1}(t)}{dt}(0)$.  Hence we can express 
	\begin{equation} \label{derinv} \dot\xi^{-1} = -  \mathrm{d}_{\xi}\rho_\xi^{-1}\lambda_{\xi^{-1}}(\dot\xi).\end{equation}
	The derivation of the identity $\xi(t)\cdot\xi(t)^{-1} = \epsilon$  at $t = 0$ implies 
	\begin{equation} \label{derinv1}\dot\xi^{-1} = -  \mathrm{d}_{\xi}\lambda_\xi^{-1}\rho_{\xi^{-1}}(\dot\xi).\end{equation}
	It follows 
	\[\mathrm{d}_{\xi}\rho_{\xi^{-1}} = \mathrm{d}_{\xi}\lambda_\xi\rho_\xi^{-1}\lambda_{\xi^{-1}},\]
	giving condition (Aii) in Proposition \ref{tangent}.\\
	Assume that $L$ has the left inverse property and differentiate the identity \[\xi(t)^{-1}\cdot\xi(t)\eta = \eta\quad \text{at}\quad t = 0.\] We obtain 
	\begin{equation} \label{derleftinv}\mathrm{d}_{\xi^{-1}}\rho_{\xi\eta}(\dot\xi^{-1}) +  \mathrm{d}_{\xi}\lambda_{\xi}^{-1}\rho_\eta(\dot\xi) = 0. \end{equation}
	Replacing (\ref{derinv}) into (\ref{derleftinv}) gives
	\[ \mathrm{d}_{\xi}\rho_{\xi\eta}\rho_\xi^{-1}\lambda_{\xi}^{-1} =  \mathrm{d}_{\xi}\lambda_{\xi}^{-1}\rho_\eta,\]
	which is equivalent to condition (Bii).\\
	Similarly, if  $L$ has the right inverse property, then it follows from the identity $\eta\xi(t)\cdot\xi(t)^{-1} = \eta$ that 
	\begin{equation} \label{derrightinv}\mathrm{d}_{\xi}\rho_\xi^{-1}\lambda_{\eta}(\dot\xi) +  \mathrm{d}_{\xi^{-1}}\lambda_{\eta\xi}(\dot\xi^{-1}) = 0.\end{equation}
	Putting (\ref{derinv1}) into (\ref{derrightinv}) we get  
	$\mathrm{d}_{\xi}\rho_\xi^{-1}\lambda_{\eta} = \mathrm{d}_{\xi}\lambda_{\eta\xi}\lambda_{\xi}^{-1}\rho_{\xi}^{-1}$, equivalently to the condition (Cii).\\
	For monoassociative loop $L$ we differentiate
	the identity $\xi(t)\cdot \xi(t)^2 = \xi(t)^2\cdot \xi(t)$ at $t = 0$. We obtain 
	\[\mathrm{d}_\xi \rho_{\xi^2}(\dot\xi) + \mathrm{d}_\xi \lambda_{\xi}\rho_{\xi}(\dot\xi) + \mathrm{d}_\xi \lambda_\xi^2(\dot\xi) = \mathrm{d}_\xi \rho_\xi^2(\dot\xi) + \mathrm{d}_\xi \rho_\xi\lambda_\xi(\dot\xi) + \mathrm{d}_\xi \lambda_{\xi^2}(\dot\xi),\]
	giving the condition (Dii).\\
	If $L$ is left alternative, right alternative or flexible then for any $\eta\in L$ we differentiate
	the identities $\xi(t)^2\cdot \eta = \xi(t)\cdot\xi(t)\eta$,\; $\eta\cdot\xi(t)^2 = \eta\xi(t)\cdot\xi(t)$, respectively, $\xi(t) \cdot \eta\xi(t)=\xi(t) \eta \cdot \xi(t)$ at 
	$t = 0$ and we obtain the conditions (Eii), (Fii), respectively, (Gii).\\
	If $L$ is a left Bol loop or right Bol loop then for any $\eta, \zeta\in L$ we differentiate  at $t = 0$ the identities $(\xi(t)\cdot\eta\xi(t))\zeta = \xi(t)(\eta\cdot\xi(t)\zeta)$, respectively, $\zeta(\xi(t)\eta\cdot\xi(t)) = (\zeta\xi(t)\cdot\eta)\xi(t)$, and we get the conditions (Hii), respectively, (Jii). Hence the assertion is true.
\end{proof}

\bigskip
\noindent
Author's addresses: \\
\'Agota Figula, Institute of Mathematics, University of Debrecen, H-4002 Debrecen, P.O.Box 400, Hungary. {\it E-mail}: {\tt {}figula@science.unideb.hu} \\[1ex]
P\'eter T. Nagy, Institute of Applied Mathematics, \'Obuda University, 1034 Budapest, B\'ecsi \'ut 96/B, Hungary. {\it E-mail}: {\tt {}nagy.peter@nik.uni-obuda.hu}
\end{document}